\documentclass[11pt]{article}

\usepackage[margin=1in]{geometry}
\usepackage{amsmath, amssymb, amsthm, mathtools}
\usepackage{bbm}
\usepackage{enumitem}
\usepackage{microtype}
\usepackage{hyperref}
\usepackage{color}

\newcommand{\cube}{\{-1,1\}^n}
\newcommand{\E}{\mathbb{E}}
\newcommand{\R}{\mathbb{R}}
\newcommand{\Z}{\mathbb{Z}}
\newcommand{\norm}[1]{\left\lVert #1\right\rVert}
\newcommand{\abs}[1]{\left\lvert #1\right\rvert}

\newtheorem{theorem}{Theorem}[section]
\newtheorem{lemma}[theorem]{Lemma}
\newtheorem{corollary}[theorem]{Corollary}
\newtheorem{proposition}[theorem]{Proposition}
\newtheorem{remark}[theorem]{Remark}
\newtheorem{property}[theorem]{Property}

{\end{list}}

\title{Dimension-free bounds for Riesz transforms on the Hamming cube via a Bellman function}

\author{Komla Domelevo, Paata Ivanisvili, Stefanie Petermichl and Alexander Volberg}
\date{\today}

\begin{document}
\maketitle

\begin{abstract}
We give a Bellman-function proof of the dimension-free estimate
\[
\Big\| \vec{R} f \Big\|_{L^p(\Omega;\,\ell^2)}
\lesssim (p-1) \,\|f\|_{L^p(\Omega)},
\qquad 2\le p<\infty,
\]
for the vector of Riesz transforms associated with the Walsh number operator on the Hamming cube $\Omega=\cube$, as well as for locally compact abelian groups, in particular $\Omega=\mathbb{Z}^n$.
The argument is based on a Poisson semigroup representation, symmetrized estimates along edges of $\Omega$, and a two-point inequality. This is the first non noncommutative proof of this result, after the seminal papers of Lust-Piquard and later Junge--Mei--Parcet.
According to an example of Lamberton, for \(1<p<2\) such a dimension-free bound is known to be false.
\end{abstract}

\tableofcontents

\section{Introduction and main results}

Dimension-free estimates for discrete Riesz transforms on product spaces go back to work of Lust-Piquard \cite{LP98,LP04}, and later extended in several works in particular in the work of Junge--Mei--Parcet \cite{JMP18}.
For instance, on the Hamming cube \(\cube\) (Walsh setting), Lust-Piquard proved dimension-free estimates for \(p\ge 2\) and pointed out, quoting a result of Lamberton, that the naive \(L^p\to L^p(\ell^2)\) dimension-free estimate fails for \(1<p<2\); see \cite{LP98}.
Those proofs relied heavily on noncommutative methods, even though the statement of the dimensionless estimate has no noncommuative elements itself. The noncommutative methods were introduced due to the inherent difficulties related to finite differences that replace continuous derivatives. Indeed, Lamberton's example for the negative result for $1<p<2$ is based on the observation that discrete derivatives of functions may have larger support than the functions themselves.

The purpose of this note is to present a Bellman-function proof for the range \(p\ge 2\), in the spirit of the Littlewood--Paley Bellman functions of Nazarov--Treil \cite{NT96} and Dragi\v cevi\'c--Volberg \cite{DV06}. To overcome the obstacle of discrete derivatives, we introduce  the new ``Two--point inequality'' method.

The difficult work of Lust-Piquard was a milestone at its time, full of novel ideas and heavily used noncommutative methods with definite obstacles for $p<2    $. It was believed for a long time that noncommutative methods could not be circumvented in this estimate, but we finally give a basic, commutative, deterministic proof for the dimensionless bound. 

We mention the generalizations of the work of Lust-Piquard as well as related results on discrete groups: in the work of Junge--Mei--Parcet \cite{JMP18} novel ideas extended the methods of Lust-Piquard and Pisier and resulted in more general estimates, including those for discrete Riesz vectors. For the second order Riesz transform the situation is simpler. Both the sharp and dimensionless estimate are available in Domelevo--Petermichl \cite{DP14}. While only for one possible definition of discrete Hilbert transform the sharp estimate is known for all $1<p<\infty$ in Ba{\~n}uelos--Kwa{\'s}nicki \cite{BK19}.  

The results presented in this paper are valid for any locally compact abelian group (see \cite{LP04} for the definition of the operators).
This includes in particular the Hamming cube $\Omega=\cube=(\Z/2\Z)^n$ and the grid $\Omega=\Z^n$, but also products of cyclic groups for example.
We only expose the proofs for the cases $\Omega=\cube$ and $\Omega=\Z^n$, but they extend \textit{mutatis mutandis} to any locally compact abelian groups.

\subsection{Discrete Riesz vector on the Hamming cube}

The geometry of the Hamming cube and the corresponding discrete operators deserve a special exposition.
We consider the Hamming cube \(\cube\) with the uniform probability measure \(\mu\) and expectation \(\E\).
For \(x=(x_1,\dots,x_n)\in\cube\) and \(1\le j\le n\) let \(x^{(j)}\) be the point obtained from \(x\) by flipping the \(j\)-th coordinate.

Define the discrete derivatives
\[
D_j f(x)\coloneqq \frac{f(x)-f\big(x^{(j)}\big)}{2},\qquad 1\le j\le n.
\]
Observe that its formal adjoint is $D_j^\ast=D_j$. It follows that the corresponding Laplacian is
\[
\Delta f \coloneqq \sum_{j=1}^n D_j^\ast D_j f = \sum_{j=1}^n  D_j f.
\]
because $D_j^2=D_j$. This operator $\Delta$ is the nonnegative (Walsh) number operator with square root $A=\Delta^{1/2}$.
We write further \(\nabla f = (D_1 f,\dots,D_n f)\) and \(\abs{\nabla f}^2=\sum_{j=1}^n \abs{D_j f}^2\).

The \(j\)-th Riesz transform is
\[
R_j \coloneqq D_j \Delta^{-1/2}
\]
and the Riesz vector is $\vec{R}=(R_1,...,R_n)$.

\subsection{Discrete Riesz vector on \texorpdfstring{$\Z^n$}{Zn}}

If $e_j=(0,\ldots,0,1,0,\ldots,0)$ denotes the $j$-th standard basis vector,
we define the partial discrete derivative in the $i$-th direction $\partial_j f$ of a function $f$ defined on $\Z^n$ as
\[ \partial_j f(x) = f(x+e_j)-f(x) , \quad\forall x\in\Z^n. \]
Its formal adjoint is
\[ \partial_j^\ast f(x) = f(x-e_j)-f(x) , \quad\forall x\in\Z^n. \]
The gradient operator is the collection of the $2n$ discrete derivatives
 $\nabla = (\partial_1 ,\partial_1^\ast,\ldots,\partial_n,\partial_n^\ast ).$
The Laplacian operator $\Delta = \sum_{j=1}^{n} \partial_j^\ast \partial_j$ is self adjoint nonnegative with square root $A=\Delta^{1/2}$. The $2n$ discrete Riesz transforms are
\[
R_j f = \nabla_j A^{-1}.
\]
Finally the Riesz vector in $\mathcal{Z}^n$ is the collection of the $2n$ operators
$$ \vec{R} = (R_1,R_2,\ldots,R_{2n})= \nabla A^{-1}.$$

\subsection{Main theorem and the \texorpdfstring{$p<2$}{p<2} obstruction}

\begin{theorem}[Dimension-free boundedness for \(p\ge 2\)]
\label{thm:main}
There exists a universal constant \(C>0\) such that for every $2\le p<\infty$,
every \(n\ge 1\), $\Omega = \cube$ or $\Omega=\Z^n$, and every \(f:\Omega\to\R\),
\[
\Big\|\vec{R} f\Big\|_{L^p(\Omega;\ell^2)}
\le C\,(p-1)\,\|f\|_{L^p(\Omega)}.
\]
\end{theorem}

The dimension-free upper bound with linear dependence on $p$ for the Walsh/Hamming-cube Riesz transforms is already implicit in Lust-Piquard’s work and was recorded explicitly in Naor’s formulation; see Lust-Piquard \cite{LP98}, Ben Efraim–Lust-Piquard \cite{BELP07}, and Naor \cite{Na16}. To the best of our knowledge, that this dependence is optimal for the Walsh/Hamming-cube is not known, and this is the object of the following Theorem, proven in Section \ref{sec:sharpness}.

\begin{theorem}[Sharpness of the $p$-dependence on the Hamming cube]
\label{thm:sharpness}
There exists a universal constant $c>0$ such that for every $2\le p<\infty$,
\[
\sup_{N\ge 1}\sup_{0\neq f:\{-1,1\}^N\to\R}
\frac{\|\vec R f\|_{L^p(\{-1,1\}^N;\ell^2_N)}}{\|f\|_{L^p(\{-1,1\}^N)}}
\ge c p.
\]
More precisely, one may take $N$ even with $N\simeq e^{2p}$ and
\[
f_N(x)=\mathbbm{1}_{\{x_1+\cdots+x_N\ge 0\}},
\qquad x\in\{-1,1\}^N.
\]
Thus the upper bound in Theorem~\ref{thm:main} has the correct order of growth in $p$.
\end{theorem}

\begin{remark}[Sharpness of the $p$-dependence on \texorpdfstring{$\mathbb Z^n$}{Zn}]
\label{rem:sharpness-Zn}

The linear dependence on $p$ is also sharp on $\mathbb Z^n$. This is already visible in dimension one. Indeed, as pointed out in Ba\~nuelos-Kwaniscki \cite{BK19}, it was shown by E. Laeng \cite{La07} that the $L^p$ norm of discrete Hilbert transforms is at least as large as the norm of the continuous Hilbert transform. This means that the dependence in $p$ is at least linear, therefore linear for $p\geq 2$ thanks to Theorem \ref{thm:main} or previous results.

Regarding sharp $L^p$-norm estimates of discrete Hilbert tranforms (not only the linear dependence in $p$), we refer to \cite{La07} for specific $p$'s, to \cite{BK19} for the discrete Hilbert transform introduced by Hilbert and all $p$'s, and to \cite{BK24} for the {R}iesz-{T}itchmarsh Hilbert transform (introduced in \cite{T26}) for even integers {$p$}.
\end{remark}

\begin{proposition}[Failure of Dimension-free boundedness for \(1<p<2\)]
\label{prop:failure}
Let $\Omega = \cube$ or $\Omega=\mathbb{Z}^n$. There is no constant \(C_p\) independent of \(n\) such that, for all $f\in L^p(\Omega)$,
\[
\Big\|\vec{R} f\Big\|_{L^p(\Omega;\ell^2)}
\le C_p\,\|f\|_{L^p(\Omega)}.
\]
\end{proposition}

This is a theorem of D.~Lamberton communicated in \cite{LP98} for the case
$\Omega=\cube$, and the example is adapted to the case $\Omega=\mathbb{Z}^n$ in \cite{LP04}.

\section{Preliminary results}

\subsection{Weak representation of the Riesz transforms -- The case \texorpdfstring{$\Omega=\cube$}{of the Hamming cube}}
In this paragraph, we make use of the Walsh system, a convenient tool for identifying functions on the Hamming cube.
For any subset $S\subset \{1,..,n\}$ define the Walsh characters \[w_S(x)=\prod_{j\in S} x_j.\] They satisfy 
\[D_j w_S = \mathbbm{1}_{\{j\in S\}}w_S\] 
and consequently  
\[\Delta w_S = |S|\,w_S.\] 
This is the arising eigenbasis for the self adjoint (compact) number operator  $\Delta$. For \(t\ge 0\) define the Poisson semigroup
\[
P_t \coloneqq e^{-t\sqrt{\Delta}}.
\]
Set \(\Delta^{-1/2}\mathbbm{1}=0\) and on the orthogonal complement of constants define \(\Delta^{-1/2}w_S=|S|^{-1/2}w_S\) for \(S\neq\varnothing\). There hold 
\[
P_t w_S=e^{-t\sqrt{|S|}}w_S,\quad \partial_t P_t w_S=-\sqrt{|S|}e^{-t\sqrt{|S|}}w_S\quad
\]
Setting $\Delta_{x,t} \coloneqq \Delta - \partial_{tt}$, we have immediately
\begin{property}\label{Pt properties Hamming Cube}
    \begin{equation}
        \Delta_{x,t} P_t = 0, \qquad \lim_{t\to\infty}  t \partial_t P_t w_S = 0.
    \end{equation}
\end{property}
The following dualized representation formula for the Riesz vector will enable our proof via Bellman function.
For \(p=2\) the operator \(R:f\mapsto (R_j f)_{j=1}^n\) is an isometry on the mean-zero subspace of \(L^2\), since
\(\sum_{j=1}^n R_j^\ast R_j = I\) on \(\mathbbm{1}^\perp\).
For \(p>2\), let \(q=p/(p-1)\in(1,2)\) and take the supremum over all \(g=(g_1,\ldots,g_n)\) with \(\|g\|_{L^q(\ell^2)}\le 1\),
\[
\Big\|\big(\vec{R}f\big)\Big\|_{L^p(\ell^2)}
=\Big\|\big(R_1 f,\dots,R_n f\big)\Big\|_{L^p(\ell^2)}
=\sup_{\|g\|_{L^q(\ell^2)}\le 1}\Big|\E\sum_{j=1}^n (R_j f)g_j\Big|.
\]
We finally express the product above in terms of derivatives in the upper half space:
\begin{lemma}[Weak representation of the Riesz transforms]
\label{lem:semigroup-identity}
Let \(f:\cube\to\R\) and \(g=(g_1,\dots,g_n):\cube\to\R^n\).
Then
\begin{equation}\label{eq:semigroup-id}
\E \sum_{j=1}^n (R_j f)\,g_j
= -4\int_0^\infty \E \sum_{j=1}^n \big(D_j P_t f\big)\,\big(\partial_t P_t g_j\big)\, t\,dt.
\end{equation}
\end{lemma}

\begin{proof}
Expand in the Walsh basis. For \(S\neq\varnothing\), observe 
\(R_j w_S=\mathbbm{1}_{\{j \in  S\}}|S|^{-1/2}w_S\).
A term-by-term computation shows that both sides are well defined and agree on each Walsh character and hence on all functions.
\end{proof}

\subsection{Weak representation of the Riesz transforms -- The case \texorpdfstring{$\Omega=\Z^n$}{of Zn}}

For $t\geq 0$, define again the Poisson semigroup $P_t\coloneqq e^{-t\sqrt{\Delta}}$.
Setting again $\Delta_{x,t} \coloneqq \Delta - \partial_{tt}$, and using Fourier transforms in $L^2(\mathbb{Z}^n)$,
we have
\begin{property}\label{Pt properties Zn}
    \begin{equation}
        \Delta_{x,t} P_t = 0, \qquad \lim_{t\to\infty}  t \partial_t P_t f = 0 \text{ in } L^2(\mathbb{Z}^n).
    \end{equation}
\end{property}
Using Fourier transforms in $L^2(\mathbb{Z}^n)$, we have again
\begin{lemma}[Weak representation of the Riesz transforms]
\label{lem:semigroup-identity bis}
Let \(f:\mathbb{Z}^n\to\R\) and \(g=(g_1,\dots,g_{2n}):\mathbb{Z}^n\to\R^{2n}\).
Denote $(\nabla_1,\nabla_2,\ldots,\nabla_{2n})=(\partial_1,\partial_1^\ast,\ldots,\partial_n,\partial_n^\ast)$.
Then
\begin{equation}\label{eq:semigroup-id bis}
\sum_{x\in\mathbb{Z}^n} \sum_{j=1}^{2 n} (R_j f)\,g_j
= -4 \sum_{x\in\mathbb{Z}^n} \int_0^\infty  \sum_{j=1}^{2 n} \big(\nabla_j P_t f\big)\,\big(\partial_t P_t g_j\big)\, t\,dt.
\end{equation}
\end{lemma}

\subsection{Integration by parts and Green's functions}

Let $Q_T$ denote the cylinder $Q_T=\Omega\times(0,T)$, with $\Omega=\cube$ or $\Omega=\Z^n$ and further let
\[
\Delta_{x,t} = \Delta + \left( -\frac{\partial^2}{\partial t^2}\right)
\]
denote the semi-discrete Laplacian, a nonnegative self-adjoint operator.
The Poisson extension $P_tf$ of \(f:\Omega\to\R\) is harmonic with $\Delta_{x,t}P_tf(x)=0$ in $Q_T$. 
Also, it is an exercise to show the integration by parts formula, for \(f:Q_T\to\R\) and \(g:Q_T\to\R\),
\begin{align}
    \mathbb{E}\int_0^T \Delta_{x,t} f \cdot g\ dt - \mathbb{E}\int_0^T f \cdot \Delta_{x,t} g\ dt 
    & = \mathbb{E}( \partial_t f(0) \cdot  g(0) - \partial_t f(T) \cdot  g(T) )  \\
    &  \qquad - \mathbb{E}( \partial_t g(0) \cdot  f(0) - \partial_t g(T) \cdot  f(T) )
\end{align}
We have used the operator $\mathbb{E}$ to denote either the expectation on $\Omega=\cube$
or the summation operator $\mathbb{E}=\sum_{x\in\Omega}$ on $\Omega=Z^n$.

Introduce now the Green's function $G^{(T)}(x,t)=G^{(T)}(t)$ solution to the Poisson problem with nonhomogeneous Dirichlet boundary conditions
\[
\begin{cases}
    \Delta_{x,t} G^{(T)}(x,t) = \delta_T(t), & (x,t)\in Q_{2T}, \\
    G(\cdot,0)=0, \quad G(\cdot,2T)=T, & x\in\Omega,
\end{cases}
\]
where $\delta_T(t)$ is the Dirac distribution at time $T$. The solution is explicitly 
\[
G^{(T)}(x,t)=G^{(T)}(t)=
\begin{cases}
    t, & 0\leq t \leq T, \\
    T, & T \leq t \leq 2T.
\end{cases}
\]
Integrating by parts the distribution $\Delta_{x,t} G^{(T)}$ against a continuous in time function $b(x,t)$ on the cylinder $Q_{T+\varepsilon}$, $0<\varepsilon<T$, yields

\begin{align}
    \mathbb{E} b(x,T) 
        & = \mathbb{E}\int_0^{T+\varepsilon} (\Delta_{x,t} G^{(T)})(x,t) \cdot b(x,t)\ dt \\
        & = \mathbb{E}\int_0^{T+\varepsilon} G^{(T)}(x,t) (\Delta_{x,t}b)(x,t)\ dt 
                + \mathbb{E} b(x,0) +  T\mathbb{E}\partial_t b(x,T+\varepsilon),
\end{align}
and letting $\varepsilon$ go to $0$ ensures
\begin{align} \label{eq: Integrated Bellman T}
    \mathbb{E} b(T)
        & = \mathbb{E}\int_0^{T} (\Delta_{x,t}b)(t)\ t\ dt
                + \mathbb{E} b(0) +  T\mathbb{E}\partial_t b(T).
\end{align}



\section{Bellman function}

From now on we fix \(p>2\) and write \(q=\frac{p}{p-1}\in(1,2)\).
Set
\[
\delta \coloneqq \frac{q(q-1)}{16}.
\]
For \(u,v\ge 0\) use the Nazarov--Treil function
\[\beta(u,v)\coloneqq u^p+v^q+\delta\gamma(u,v)\] with 
\[\gamma(u,v)\coloneqq \,
\begin{cases}
u^2 v^{2-q},& u^p\le v^q,\\[2mm]
\frac{2}{p}u^p+\Big(\frac{2}{q}-1\Big)v^q,& u^p\ge v^q.
\end{cases}
\]

\begin{proposition}
  $u^p + v^q \leqslant \beta (u, v) \leqslant (1 + \delta) (u^p + v^q) .$
\end{proposition}

\begin{proof}
This size propery of $\beta$ is in \cite{NT96} and included for convenience.
  Since $2 \leqslant p < \infty$ and $1 < q \leqslant 2$, $\gamma (u, v)
  \geqslant 0$ and thus $\beta (u, v) \geqslant u^p + v^q .$ If $u^p \leqslant v^q$ then $u \leqslant v^{q - 1}$ and thus $u^2 v^{2 - q}
  \leqslant v^q$ and the estimate follows. If $u^p \geqslant v^q$ the estimate follows from the range for $p$ and $q$:
  $\frac{2}{p} \leqslant 1$ and $\frac{2}{q} - 1 \leqslant 1$.
\end{proof}

From {\cite{NT96}} it is known

\begin{proposition}
  The function $\beta$ is in $C^1$ and piecewise in $C^2$.
\end{proposition}
In \cite{NT96} several properties have been shown, for instance the convexity. Due to difficulties relating to the discrete spaces, we will need to show specific properties, many of which could be found in the literature, perhaps with slight variations. For readability and self containment the simple calculations are included.

\begin{proposition}
  There hold
  \begin{enumerate}
    \item Provided $v \neq 0$ and $u \neq 0$,
    \[ \beta_{22}'' (u, v) > 0, \beta_{11}'' (u, v) > 0. \]
    \item Provided $v \neq 0$,
    \[ \left( \beta_{11}'' - \frac{{\beta_{12}''}^2}{\beta_{22}''} \right)
       (u, v) \geqslant p (p - 1) u^{p - 2} + \delta v^{2 - q} . \]
    \item Provided $v \neq 0$ and $u \neq 0$,
    \[ \left( \beta_{22}'' - \frac{{\beta_{12}''}^2}{\beta_{11}''} \right)
       (u, v) \geqslant \frac{7}{8} (q - 1) v^{q - 2}, \quad \frac{\beta_2'
       (u, v)}{v} \geqslant (q - 1) v^{q - 2} . \]
  \end{enumerate}
  
\end{proposition}

\begin{proof}
  We compute the derivatives of $\beta$.
  \begin{eqnarray*}
    \beta_1' (u, v) & = & p u^{p - 1} + \delta \left\{ \begin{array}{ll}
      2 u v^{2 - q} & : u^p \leqslant v^q\\
      2 u^{p - 1} & : u^p \geqslant v^q
    \end{array} \right.\\
    \beta_2' (u, v) & = & q v^{q - 1} + \delta \left\{ \begin{array}{ll}
      u^2 (2 - q) v^{1 - q} & : u^p \leqslant v^q\\
      (2 - q) v^{q - 1} & : u^p \geqslant v^q
    \end{array} \right.\\
    \beta_{11}'' (u, v) & = & p (p - 1) u^{p - 2} + \delta \left\{
    \begin{array}{ll}
      2 v^{2 - q} & : u^p \leqslant v^q\\
      2 (p - 1) u^{p - 2} & : u^p \geqslant v^q
    \end{array} \right.\\
    \beta_{22}'' (u, v) & = & q (q - 1) v^{q - 2} + \delta \left\{
    \begin{array}{ll}
      - u^2 (2 - q) (q - 1) v^{- q} & : u^p \leqslant v^q\\
      (2 - q) (q - 1) v^{q - 2} & : u^p \geqslant v^q
    \end{array} \right.\\
    \beta_{12}'' (u, v) & = & \delta \left\{ \begin{array}{ll}
      2 (2 - q) u v^{1 - q} & : u^p \leqslant v^q\\
      0 & : u^p \geqslant v^q
    \end{array} \right.
  \end{eqnarray*}
For {\it{(i)}}: The estimate $\beta_{11}'' (u, v) > 0$ if $v \neq 0$ and $u \neq 0$ is obvious. 
If $u^p \geqslant v^q$ then obviously $\beta_{22}'' (u, v) > 0$ if $v \neq 0$.
  If $u^p \leqslant v^q$ then $u^2 v^{- q} \leqslant v^{q - 2}$ and using
  $\delta = \frac{q (q - 1)}{16}$ yields if $v \neq 0$
  \[ \beta_{22}'' (u, v) \geqslant (q - 1) q \big( 1 - \frac{(2 - q) (q -
     1)}{16} \big) v^{q - 2} > 0. \]
For {\it{(ii)}}:  If $u^p \leqslant v^q$ then {\it{(ii)}} is implied if we show $\delta
  v^{2 - q} \beta_{22}'' (u, v) {\geqslant \beta_{12}''}^2 (u, v)$. But
  using $u^2 v^{2 - 2 q} \leqslant 1$ as well as $\delta = \frac{q (q -
  1)}{16}$,
  \begin{align*}
    \delta v^{2 - q} \beta_{22}'' (u, v) {- \beta_{12}''}^2 (u, v)  = &
    \delta q (q - 1) - \delta^2 (2 - q) ((q - 1) + 4 (2 - q)) u^2 v^{2 - 2
    q}\\
     \geqslant & \delta q (q - 1) - \delta^2 (2 - q) ((q - 1) + 4 (2 - q))\\
     = & \delta q (q - 1) \big( 1 - \frac{1}{16} (2 - q) (- 3 q + 7)
    \big)\\
     \geqslant & 0.
  \end{align*}
  If $u^p \geqslant v^q$ then $u^{p - 2} \geqslant v^{2 - q}$ and since
  $\beta_{12}'' (u, v) = 0$ this implies {\it{(ii)}}.
  
  \noindent{For} {\it{(iii)}}:
  If $u^p \leqslant v^q$ then $u^2 v^{- q} \leqslant v^{q - 2}$. There holds
  $\beta_{11}'' (u, v) \geqslant 2 \delta v^{2 - q}$ and so
  \[ \frac{{\beta_{12}''}^2}{\beta_{11}''} \leqslant 2 \delta (2 - q)^2 u^2
     v^{- q}, \]
  thus
  \begin{align*}
    \big( \beta_{22}'' - \frac{{\beta_{12}''}^2}{\beta_{11}''} \big) (u,
    v)  \geqslant & q (q - 1) v^{q - 2} - (2 - q) (3 - q) \delta u^2 v^{-
    q}\\
     \geqslant & q (q - 1) \big( 1 - \frac{1}{16} (2 - q) (3 - q) \big)
    v^{q - 2}\\
     \geqslant & \frac{7}{8} (q - 1) v^{q - 2}.
  \end{align*}
  Further, $v^{- 1} \beta_2' (u, v) \geqslant q v^{q - 2} \geqslant (q - 1)
  v^{q - 2}$.
  If $u^p \geqslant v^q$ then $\beta_{12}'' (u, v) = 0$ and $\beta_{22}''
  (u, v) = (q - 1) q \left( 1 + \frac{1}{16} (2 - q) (q - 1) \right) v^{q - 2}
  \geqslant (q - 1) v^{q - 2}$. Obviously $v^{- 1} \beta_2' (u, v)
  \geqslant q v^{q - 2} \geqslant (q - 1) v^{q - 2}$.
\end{proof}

\begin{proposition}
\label{propertiesB}
Fix $2\le p<\infty, q=p/(p-1)$. 
For \((s,y)\in\R\times\R^n\) define the Bellman function
\[
B(s,y)\coloneqq \beta\big(\abs{s},\abs{y}\big).
\]
Finally define the weight
\begin{equation}\label{eq:tau-def}
\tau(y)\coloneqq 
\abs{y}^{2-q}
\end{equation} and recall 
\[
\delta \coloneqq \frac{q(q-1)}{16}=\frac{q}{16(p-1)}.
\]

For all \((s,y)\in\R\times\R^n\), there hold
\[
0 \le B(s,y)\le (1+\delta)\big(\abs{s}^p+\abs{y}^q\big).\]

\[d^2B(s,y)\ge \frac{1}{16(p-1)}\left(\begin{array}{ll}\tau(y)&0\\
0&\tau(y)^{-1}I_n
\end{array}
\right) 
\]
\[
B'_{y_i}y_i\ge 0 \; \forall i=1,...,n.
\]

\end{proposition}

\begin{proof}
The size estimates follow immediately. To see the estimate on the Hessian, we observe that $B$ is symmetric in $s$ and radial in $y$. This implies for  $s,y\neq 0$  that 
\[B'_1(s,y)=\beta'_1(|s|,|y|)\text{sign}(s), B_{2}'(s,y)=\nabla_2 B(s,y)=\beta'_2(|s|,|y|)\frac{y}{|y|} ,\] 
 \[
  B_{11}''(s,y)=\beta''_{11}(|s|,|y|), B''_{12}(s,y)=\partial_1\nabla_2B(s,y)=\langle \beta''_{12}(|s|,|y|),\frac{y}{|y|}\rangle\text{sign}(s).
 \]
 Having a closer look at the Hessian in variable 2 $d^2B_22=B''_{22}$ yields 
 \[
 B'_{y_i}(s,y)=\beta'_2(|s|,|y|)\frac{y_i}{|y|}.
 \]
 If $i\neq j$ then 
 \[
 B''_{y_iy_j}(s,y)=\beta''_{22}(|s|,|y|)\frac{y_iy_j}{|y|^2}-\beta'_2(|s|,|y|)\frac{y_iy_j}{|y|^3}.
 \]
 If $i=j$ then 
 \[
 B''_{y_i^2}(s,y)=\beta''_{22}(|s|,|y|)\frac{y_i^2}{|y|^2}+\beta'_2(|s|,|y|)\frac{|y|^2-y_i^2}{|y|^3}.
 \]
 Write $\mathbb{R}^n=Y \oplus Y^{\perp}$, where $Y=\text{span}\{y\}$. Computing the $i$th component 
 \[
 (B''_{22}(s,y)y)_i=\beta''_{22}(|s|,|y|)y_i
 \] and for $b\perp y$ we compute 
 \[
 (dB''_{22}(s,y)b)_i=\beta'_2(|s|,|y|)\frac1{|y|}b_i.
 \]
 We deduce that $y$ is Eigenvector of $d^2B_{22}(s,y)$ with Eigenvalue $\beta''_{22}(|s|,|y|)$. Its orthogonal complement is Eigenspace to the Eigenvalue $\beta'_2(|s|,|y|)\frac1{|y|}$. Writing $\xi=\lambda_{\parallel}\frac{y}{|y|}+\xi_{\perp}$ then 
 \[
 \langle B''_{22}(s,y)\xi,\xi \rangle = \beta''_{22}(|s|,|y|)\lambda_{\parallel}^2+\beta'_2(|s|,|y|)\frac1{|y|}|\xi_{\perp}|^2.
 \]
 To obtain an estimate for the full Hessian, we compute 
 \begin{align*}
 \langle d^2B(s,y)
 (\sigma,\xi)^t,(\sigma,\xi)^t
\rangle
 & = B''_{11}(s,y)\sigma^2+2B''_{12}(s,y)\sigma\xi+\langle B''_{22}(s,y)\xi,\xi\rangle \\
 & = \beta''_{11}(|s|,|y|)\sigma^2 +2\beta''_{12}(|s|,|y|) \text{sign}(s)\sigma \lambda_{\parallel} \\
 &
  \qquad\qquad\qquad +\lambda_{\parallel}^2\beta''_{22}(|s|,|y|)+ \beta'_2(|s|,|y|)\frac{|\xi_{\perp}|^2}{|y|}.
 \end{align*}
 Using $\beta''_{22}>0$ and $\beta'_2\ge 0$ and minimizing in $\lambda_{\parallel}$ yields 
 \[
 \langle d^2B(s,y)
 (\sigma,\xi)^t,(\sigma,\xi)^t
\rangle \ge \big( \beta''_{11} - \frac{\beta''^2_{12}}{\beta''_{22}}\big)(|s|,|y|)\sigma^2 \ge p(p-1)|s|^{p-2}\sigma^2+\delta |y|^{2-q}\sigma^2\ge \frac{1}{16(p-1)}|y|^{2-q}\sigma^2.
 \]
 Using $\beta''_{11},\beta''_{22}> 0$ and minimizing in $\sigma$ yields 
 \[
 \langle d^2B(s,y)
 (\sigma,\xi)^t,(\sigma,\xi)^t
\rangle \ge \big( \beta''_{22}-\frac{\beta''^2_{12}}{\beta''_{11}}\big) \lambda^2_{\parallel}+\beta'_2\frac{|\xi_{\perp}|^2}{|y|}\ge \frac{7}{8}(q-1)|y|^{q-2}|\xi|^2\ge \frac{7}{8(p-1)}|y|^{q-2}|\xi|^2.
 \]
 
\end{proof}

\section{Bellman energy and Bilinear estimates}

\subsection{Dissipation of the Bellman energy}
Let successively $v(x,t)\coloneqq(f(x,t),g(x,t))\coloneqq (P_t f(x),P_t g(x))$, and $b(x,t)\coloneqq (B\circ v)(x,t)$.
Using Properties \ref{Pt properties Hamming Cube} and \ref{Pt properties Zn} for the decay properties of the Poisson extensions, using the positivity, boundedness and regularity of the Bellman function,
together with standard approximations and density arguments, we can pass to the limit $T$ goes to infinity
in \eqref{eq: Integrated Bellman T} and get
\begin{align} \label{eq: Integrated Bellman infinity}
        - \mathbb{E}\int_0^{\infty} (\Delta_{x,t}b)(t)\ t\ dt
             \leq     \mathbb{E} b(0) 
\end{align}
Here and in what follows, we assume that the Bellman function is $\mathcal{C}^2$ instead of $\mathcal{C}^1$ and only piecewise $\mathcal{C}^2$. This is justified after the use of standard regularization techniques, where we replace the Bellman function $B$ by a regularized version $B_\eta$ at the cost of a small loss in the estimates. The loss goes to zero as $\eta$ goes to zero. See for example \cite[Lemma 1.2]{DV06} for details and applications to bilinear imbeddings.

Our goal is therefore to estimate $-\mathbb{E}\Delta_{x,t}b(t)$. 
For the contribution of the time derivative, expand simply
\[
\partial^2_t b(x,t) = \nabla B(v(x,t))\cdot \partial^2_t v(x,t) + \langle d^2B(v(x,t)) \partial_t v(x,t),\partial_tv(x,t)\rangle.
\]
By $\Omega$ we denote as before either $\cube$ or $\mathbb{Z}^n$. Observe now that the discrete Laplacian $-\Delta$ on $\Omega$ is (up to a constant in $\cube$) nothing but the graph Laplacian $\Delta^{\mathcal{G}}$ defined below.
The vertices of this graph are the points $x\in\Omega$, and the oriented edges are the couples $(x,y)$ of neighboring vertices.
For any function $v:\Omega\to\R^p$, any vertex $x\in\Omega$, we have
\[
\Delta^{\mathcal{G}} v(x) = \sum_{y\sim x} ( v(y) - v(x) ),
\]
where the sum is taken over all neighboring vertices $y$ of $x$. This means $n$ neighbors in $\cube$ and $2n$ neighbors in $\mathbb{Z}^n$.
Recalling that $b=B\circ v$, we expand
\begin{align}
    (\Delta^{\mathcal{G}} b)(x,t)
        & = \sum_{y\sim x} b(y,t) - b(x,t) \\
        & = \sum_{y\sim x} \nabla B(v(x,t))\cdot (v(y,t)-v(x,t)) \\
        &    \qquad      + \sum_{y\sim x} \int_0^1 \langle d^2B(v(x,y,t,\theta)) (v(y,t)-v(x,t)) , (v(y,t)-v(x,t)) \rangle   (1-\theta) d\theta \\
        & \coloneqq \nabla B(v(x,t)) \cdot (\Delta v)(x,t) \quad + \quad \sum_{y\sim x} \alpha(x,y).
\end{align}
where we have noted $v(x,y,t,\theta)=(1-\theta) v(x) + \theta v(y)$.
In order to estimate the contribution of the terms $\alpha(x,y)$ associated to each oriented edge, we symmetrize, after summing over $\Omega$,
\[
\sum_{x\in\Omega}\sum_{y\sim x} \alpha(x,y)
    = \sum_{(x,y)} \alpha(x,y) + \alpha(y,x)
    \coloneqq \sum_{(x,y)} 2 \bar{\alpha}(x,y)
    = \sum_{x\in\Omega}\sum_{y\sim x} \bar{\alpha}(x,y),
\]
where $\bar{\alpha}(x,y)=\bar{\alpha}(y,x)$ is the average of $\alpha(x,y)$ and $\alpha(y,x)$.
It is easy to check that
\begin{align}
\bar{\alpha}(x,y)
    & = \frac{1}{2}\int_0^1 \Big\langle d^2B(v(x,y,t,\theta)) \ (v(y,t)-v(x,t)) , (v(y,t)-v(x,t)) \Big\rangle (1-\theta+\theta) d\theta \\
    & = \frac{1}{2} \Big\langle  \Big(\int_0^1d^2B(v(x,y,t,\theta))d\theta \Big)  \ (v(y,t)-v(x,t)) , (v(y,t)-v(x,t)) \Big\rangle \\
    & = \frac{1}{2} \Big\langle  \overline{d^2B}(v(x,t),v(y,t)) \ (v(y,t)-v(x,t)) , (v(y,t)-v(x,t)) \Big\rangle,
\end{align}
that is only the \emph{uniform} average $\overline{d^2B}(v_1,v_2)$ of the Hessian $d^2B(v)$ on edges $(v_1,v_2)$ contributes to the dissipation of the Bellman function.
Collecting time and space derivatives, and using that $\Delta_{x,t}v=0$, we finally get
\begin{align}
-\sum_{x\in\Omega}\Delta_{x,t}b (x,t)
& = \frac{1}{2} \sum_{x\in\Omega}
        \sum_{y\sim x}   \langle \overline{d^2B}(v(x,t),v(y,t)) \ (v(y,t)-v(x,t)) , (v(y,t)-v(x,t)) \rangle \label{eq: Delta b} \\
& \qquad    + \sum_{x\in\Omega}  \langle d^2B(v(x,t)) \partial_t v(x,t),\partial_tv(x,t)\rangle. \label{eq: dt b}
\end{align}

\subsection{Two-point inequality}\label{two2}

As usual, the bilinear imbedding discussed in the next section is a consequence of the dissipation properties of the Bellman function.
Proposition \ref{propertiesB} provides the pointwise bounds for the Hessian:
\[
\langle d^2 B(v) dv,dv \rangle \gtrsim (q-1)  \big( |g|^{2-q} |df|^2 +  |g|^{q-2} |dg|^2 \big).
\]
However, we have seen above that in the case of discrete derivatives we need to estimate the dissipation properties of averages of the type $\overline{d^2 B}(v_1,v_2)$.
For $dv=v_2-v_1$ given, averaging over the segment $(v_1,v_2)$ the inequality above yields immediately
\[
\langle \overline{d^2 B}(v_1,v_2) (v_2-v_1),(v_2-v_1) \rangle
\gtrsim (q-1) \big( \overline{\alpha}(g_1,g_2) |f_2-f_1|^2 + \overline{\beta}(g_1,g_2) |g_2-g_1|^2 \big),
\]
with $\alpha(g) = |g|^{2-q}$ and $\beta(g) = |g|^{q-2}$. We will need
\begin{lemma}[Segment average of powers]\label{lem:segment-average}
Let \(r > 0\) and \(a,b\in\R^m\). Then
\[
\int_0^1 \norm{(1-t)a+tb}^r\,dt \;\ge\; \frac{\norm{a}^r+\norm{b}^r}{2(r+1)}.
\]
\end{lemma}

\begin{proof}
By the reverse triangle inequality,
\(\norm{(1-t)a+tb}\ge \big|(1-t)\norm{a}-t\norm{b}\big|\).
Thus it suffices to prove the scalar case with \(A=\norm{a}\), \(B=\norm{b}\):
\[
\int_0^1 \big|(1-t)A-tB\big|^r dt \ge \frac{A^r+B^r}{2(r+1)}.
\]
A direct computation (splitting at the zero of \((1-t)A-tB\)) yields
\[
\int_0^1 \big|(1-t)A-tB\big|^r dt
= \frac{A^{r+1}+B^{r+1}}{(r+1)(A+B)}.
\]
Finally,
\[
\frac{A^{r+1}+B^{r+1}}{A+B}\ge \frac{A^r+B^r}{2}
\iff A^{r+1}+B^{r+1}\ge A^rB+AB^r,
\]
which holds since \((A-B)(A^r-B^r)\ge 0\) for \(r>0\).
\end{proof}

With $0 \leq r=2-q \leq 1$, we can estimate $\overline{\alpha}(g_1,g_2) \geq \frac{1}{4} (|g_1|^{2-q} +  |g_2|^{2-q}) $.
Discarding the nonnegative contribution of $\overline{\beta}(g_1,g_2) |g_2-g_1|^2$ yields immediately
\begin{corollary}[Two-point inequality]\label{Cor: two point inequality} Let $p\geq 2$. We have
\[
\langle \overline{d^2 B}(v_1,v_2) (v_2-v_1),(v_2-v_1) \rangle
    \gtrsim (q-1)  \big( |g_1|^{2-q} +  |g_2|^{2-q}  \big) |f_2 - f_1|^2.
\]
\end{corollary}

\subsection{Bilinear imbedding}

We turn now to the bilinear imbedding. Our goal is to prove
\begin{theorem}
\label{T: bilinear imbedding}
Assume $p\geq 2$. Then
\[
\mathbb{E}\int_0^\infty \|  \nabla_{x,t} f \| \|  \nabla_{t} g \| t dt
    \lesssim
    (p-1) \| f\|_p   \|g\|_q
\]
\end{theorem}
Here we note $\nabla_{x,t}f(x,t) \coloneqq (\nabla,\partial_t)f(x,t)$ the collection of all space and time derivatives of $f$ at the point $(x,t)$. The vector of space derivatives $\nabla f(x,t)$ is the collection of all increments along edges neigboring $x$, that is $\nabla f(x,t) = ( f(y,t) - f(x,t) )_{y\sim x}$.
Notice that in the Theorem above discrete derivatives are only present for the the function $f$, whose $L^p$ norm, $p\geq 2$, is used, whereas we can only estimate time derivatives of the function $g$ as soon as $q<2$. This is in contrast with the continuous case (see e.g. \cite{DV06}) where the full gradients of both $f$ and $g$ can be estimated.

\begin{proof}[Proof of Theorem \ref{T: bilinear imbedding}]

We estimate now $-\mathbb{E}\Delta_{x,t}b(t)$ as written in \eqref{eq: Delta b}-\eqref{eq: dt b}.
Regarding the contributions of the continuous time derivatives in \eqref{eq: dt b}, we simply use the dissipation estimates of the Bellman function to get

\begin{align}
    \eqref{eq: dt b}
    & = \sum_{x\in\Omega}  \langle d^2B(v(x,t)) \partial_t v(x,t),\partial_tv(x,t)\rangle \\
    & \gtrsim (q-1)\sum_{x\in\Omega} |g(x,t)|^{2-q} (\partial_t f)^2(x,t) + |g(x,t)|^{q-2} (\partial_t g(x,t))^2(x,t).
\end{align}

For the contributions of the discrete space derivatives in \eqref{eq: Delta b}, we use the two-point inequality from Corollary \ref{Cor: two point inequality} and distribute every term $(|g(x,t)|^{2-q} + |g(y,t)|^{2-q}) (f(y,t) - f(x,t))^2$ onto their respective vertices $x$ and $y$ to get

\begin{align}
    \eqref{eq: Delta b}
    & = \frac{1}{2} \sum_{x\in\Omega}
        \sum_{y\sim x}   \langle \overline{d^2B}(v(x,t),v(y,t)) \ (v(y,t)-v(x,t)) , (v(y,t)-v(x,t)) \rangle \\
    & \gtrsim (q-1) \sum_{x\in\Omega} \sum_{y\sim x}
    (|g(x,t)|^{2-q} + |g(y,t)|^{2-q}) (f(y,t) - f(x,t))^2 \\
    & \gtrsim (q-1) \sum_{x\in\Omega} \sum_{y\sim x}
    |g(x,t)|^{2-q} (f(y,t) - f(x,t))^2 \\
    & \approx (q-1) \sum_{x\in\Omega} |g(x,t)|^{2-q} |\nabla f(x,t)|^2.
\end{align}

This allows us to estimate
\begin{align}
-\mathbb{E}\int_0^\infty \Delta_{x,t}b(x,t)\ t dt
& \geq \gtrsim (q-1) \mathbb{E} \int_0^\infty |g(x,t)|^{2-q} |\nabla_{x,t} f(x,t)|^2 + |g(x,t)|^{q-2} (\partial_t g)^2(x,t)  \ tdt \\
& \gtrsim (q-1) \mathbb{E} \int_0^\infty |\nabla_{x,t} f(x,t)| |\partial_t g(x,t)| \ t dt
\end{align}
where we used $ |g(x,t)|^{2-q} + |g(x,t)|^{q-2} \geq 2$.

Finally, recalling that 
\[
-\mathbb{E}\int_0^\infty \Delta_{x,t}b(x,t)\ t dt
\leq \mathbb{E}b(0)
\leq (1+\delta) ( \|f\|_p^p + \|g\|_q^q )
\leq 2 ( \|f\|_p^p + \|g\|_q^q ),
\]
we get
\[
\mathbb{E} \int_0^\infty |\nabla_{x,t} f(x,t)| |\partial_t g(x,t)| \ t dt
\lesssim \frac{1}{q-1} ( \|f\|_p^p + \|g\|_q^q )
\lesssim (p-1) \|f\|_p \|g\|_q,
\]
after using $ q-1 = (p-1)^{-1}$, and a good $\lambda$ inequality.
\end{proof}

\subsection{Proof of Theorem \ref{thm:main}}
We now collect easily
\begin{align}
    \Big| \mathbb{E} \big(\vec{R} f \cdot \vec{g} \big) \Big|
    & = 4 \Big| \sum_j \sum_x\int_0^\infty \partial_j f(x,t)\ \partial_t g_j \ t dt \Big| \\
    & \leq 4 \mathbb{E} \int_0^\infty \Big(\sum_j (\partial_j f)^2(x,t)\Big)^{1/2}
                                \Big(\sum_j (\partial_t g_j)^2(x,t)\Big)^{1/2}  \ t dt \\
    & \leq 1024 (p-1) \|f\|_{L^p} \|\vec{g}\|_{L^q(\ell^2)}
\end{align}
where we used successively the weak representation of the Riesz transforms, Cauchy-Schwarz, and the bilinear imbedding. This concludes the proof of Theorem \ref{thm:main}.

\section{Sharpness of the linear \texorpdfstring{$p$}{p}-dependence on the Hamming cube}
\label{sec:sharpness}

In this section we prove Theorem~\ref{thm:sharpness}. The example is the majority function. Throughout the section the dimension of the cube is denoted by $N$ rather than $n$.

Let
\[
\Omega_N=\{-1,1\}^N,
\qquad
\mu_N=\text{the uniform probability measure on }\Omega_N .
\]
The Walsh number operator on $\Omega_N$ is still denoted by $\Delta$, and its heat semigroup by
\[
H_t=e^{-t\Delta},\qquad t\ge0.
\]
On Walsh characters $w_S(x)=\prod_{i\in S}x_i$ we have
\[
H_t w_S=e^{-t|S|}w_S,
\qquad
D_jw_S=\mathbbm{1}_{\{j\in S\}}w_S.
\]
Therefore, for every function $f:\Omega_N\to\R$,
\begin{equation}\label{eq:sharp-heat-representation}
R_j f
=
D_j\Delta^{-1/2}f
=
\frac1{\sqrt\pi}\int_0^\infty t^{-1/2}D_jH_tf\,dt.
\end{equation}
Indeed, this follows term by term from the identity
\[
\lambda^{-1/2}=\frac1{\sqrt\pi}\int_0^\infty t^{-1/2}e^{-t\lambda}\,dt,
\qquad \lambda>0,
\]
and the constant Walsh component is killed by $D_j$.

We also use the standard noise representation of $H_t$:
\begin{equation}\label{eq:sharp-noise-representation}
H_tf(x)=\mathbb E_\eta f(x_1\eta_1,\ldots,x_N\eta_N),
\end{equation}
where $\eta_1,\ldots,\eta_N$ are independent signs with
\[
\mathbb P(\eta_i=1)=\frac{1+e^{-t}}2,
\qquad
\mathbb P(\eta_i=-1)=\frac{1-e^{-t}}2.
\]
We write
\begin{equation}\label{eq:sharp-s-def}
s=s(t)=\frac{1-e^{-t}}2.
\end{equation}
For $0<t\le1$,
\begin{equation}\label{eq:sharp-s-bounds}
\frac t4\le s(t)\le \frac t2.
\end{equation}

\begin{lemma}[A binomial overlap estimate]\label{lem:binomial-overlap}
There is a universal constant $c>0$ such that for all $M\ge1$, all $0<s\le1/2$, and independent random variables
\[
X\sim \operatorname{Bin}(M,s),
\qquad
Y\sim \operatorname{Bin}(M+1,s),
\]
one has
\[
\mathbb P(X=Y)\ge \frac{c}{\sqrt{1+Ms}}.
\]
\end{lemma}

\begin{proof}
Write $Y=Y_0+Z$, where
\[
Y_0\sim \operatorname{Bin}(M,s),
\qquad
Z\sim \operatorname{Bernoulli}(s),
\]
and where $X,Y_0,Z$ are independent. Then
\[
\mathbb P(X=Y)
\ge
\mathbb P(X=Y_0,\ Z=0)
=
(1-s)\mathbb P(X=Y_0)
\ge
\frac12\mathbb P(X=Y_0).
\]
It remains to estimate the collision probability of two independent copies of $\operatorname{Bin}(M,s)$.
Let $X'$ be an independent copy of $X$ and put $\lambda=Ms$.

If $\lambda\le1$, then
\[
\mathbb P(X=X')\ge \mathbb P(X=0,X'=0)=(1-s)^{2M}.
\]
Since $s\le1/2$, $\log(1-s)\ge -2s$, and hence
\[
(1-s)^{2M}\ge e^{-4Ms}=e^{-4\lambda}\ge e^{-4}.
\]
This gives the desired bound when $\lambda\le1$.

Assume now that $\lambda\ge1$. Let
\[
I=\{k\in\mathbb Z: |k-\lambda|\le2\sqrt\lambda\}.
\]
Since $\operatorname{Var}(X)=Ms(1-s)\le\lambda$, Chebyshev's inequality gives
\[
\mathbb P(X\in I)\ge \frac34.
\]
Writing $p_k=\mathbb P(X=k)$ and using Cauchy--Schwarz,
\[
\mathbb P(X=X')
=\sum_k p_k^2
\ge \sum_{k\in I}p_k^2
\ge \frac{\left(\sum_{k\in I}p_k\right)^2}{|I|}.
\]
Moreover $|I|\le4\sqrt\lambda+1\le5\sqrt\lambda$. Hence
\[
\mathbb P(X=X')\ge \frac{(3/4)^2}{5\sqrt\lambda}
=\frac{9}{80\sqrt\lambda}.
\]
Combining the two cases yields
\[
\mathbb P(X=X')\ge \frac{c}{\sqrt{1+Ms}},
\]
and the preceding comparison between $Y$ and $Y_0$ proves the lemma.
\end{proof}

We now turn to the majority function. Let $N=2m$ be even and set
\[
f_N(x)=\mathbbm{1}_{\{x_1+\cdots+x_N\ge0\}},
\qquad x\in\Omega_N.
\]
Let
\[
B_N=\left\{x\in\Omega_N:\sum_{i=1}^N x_i=0\right\}
\]
be the central layer of the cube.

\begin{lemma}\label{lem:majority-coordinate-lower-bound}
There is a universal constant $c>0$ such that, for all sufficiently large even $N$, every $x\in B_N$, and every coordinate $j$ with $x_j=1$,
\[
R_jf_N(x)\ge c\frac{\log N}{\sqrt N}.
\]
\end{lemma}

\begin{proof}
Fix $x\in B_N$ and $j$ such that $x_j=1$. For the noise variables in \eqref{eq:sharp-noise-representation}, put
\[
A=\sum_{i\neq j}x_i\eta_i.
\]
Then
\[
H_tf_N(x)=\mathbb E\,\mathbbm{1}_{\{A+\eta_j\ge0\}},
\qquad
H_tf_N(x^{(j)})=\mathbb E\,\mathbbm{1}_{\{A-\eta_j\ge0\}}.
\]
Therefore
\begin{align*}
2D_jH_tf_N(x)
&=H_tf_N(x)-H_tf_N(x^{(j)})  \\
&=\mathbb E\left(\mathbbm{1}_{\{A+\eta_j\ge0\}}-\mathbbm{1}_{\{A-\eta_j\ge0\}}\right).
\end{align*}
The random variable $A$ is a sum of $N-1$ signs and is therefore odd. If $\eta_j=1$, then
\[
\mathbbm{1}_{\{A+1\ge0\}}-\mathbbm{1}_{\{A-1\ge0\}}
=\mathbbm{1}_{\{A=-1\}}.
\]
If $\eta_j=-1$, then
\[
\mathbbm{1}_{\{A-1\ge0\}}-\mathbbm{1}_{\{A+1\ge0\}}
=-\mathbbm{1}_{\{A=-1\}}.
\]
Since $\mathbb E\eta_j=e^{-t}$ and $\eta_j$ is independent of $A$, we get
\begin{equation}\label{eq:sharp-DjHt-majority}
D_jH_tf_N(x)=\frac{e^{-t}}2\mathbb P(A=-1).
\end{equation}

Because $x\in B_N$ and $x_j=1$, among the coordinates $i\neq j$ there are $m-1$ indices with $x_i=1$ and $m$ indices with $x_i=-1$. Let $U$ be the number of flips among the $m-1$ positive coordinates, and let $V$ be the number of flips among the $m$ negative coordinates. Then
\[
U\sim \operatorname{Bin}(m-1,s),
\qquad
V\sim \operatorname{Bin}(m,s),
\]
independently, with $s=s(t)$ as in \eqref{eq:sharp-s-def}. Starting from
\[
\sum_{i\neq j}x_i=(m-1)-m=-1,
\]
each positive-coordinate flip decreases $A$ by $2$, while each negative-coordinate flip increases $A$ by $2$. Hence
\[
A=-1-2U+2V.
\]
Thus
\[
\{A=-1\}=\{U=V\}.
\]

We restrict $t$ to the interval
\[
\frac{16}{N}\le t\le \frac14.
\]
For $N$ sufficiently large this interval is nonempty. By \eqref{eq:sharp-s-bounds},
\[
\frac t4\le s\le \frac t2.
\]
Put $M=m-1$. Then $M\ge N/4$ for $N\ge4$. On the above range of $t$,
\[
Ms\ge \frac N4\cdot\frac t4\ge1,
\qquad
Ms\le \frac N2\cdot\frac t2=\frac{Nt}{4}.
\]
Lemma~\ref{lem:binomial-overlap} gives
\[
\mathbb P(A=-1)=\mathbb P(U=V)
\ge \frac{c}{\sqrt{Ms}}
\ge \frac{c}{\sqrt{Nt}}.
\]
Since $e^{-t}\ge e^{-1/4}$ on this interval, \eqref{eq:sharp-DjHt-majority} yields
\begin{equation}\label{eq:sharp-DjHt-lower}
D_jH_tf_N(x)\ge \frac{c}{\sqrt{Nt}},
\qquad \frac{16}{N}\le t\le\frac14.
\end{equation}
Finally, by the heat representation \eqref{eq:sharp-heat-representation},
\begin{align*}
R_jf_N(x)
&=\frac1{\sqrt\pi}\int_0^\infty t^{-1/2}D_jH_tf_N(x)\,dt \\
&\ge \frac{c}{\sqrt N}\int_{16/N}^{1/4}\frac{dt}{t}
\ge c\frac{\log N}{\sqrt N},
\end{align*}
after decreasing $c$ if necessary. This proves the lemma.
\end{proof}

\begin{proof}[Proof of Theorem~\ref{thm:sharpness}]
Let $N=2m$ be even and let $f_N$ be the majority function above. For every $x\in B_N$, exactly $N/2$ coordinates satisfy $x_j=1$. Lemma~\ref{lem:majority-coordinate-lower-bound} therefore implies
\[
|\vec R f_N(x)|
=\left(\sum_{j=1}^N |R_jf_N(x)|^2\right)^{1/2}
\ge
\left(\frac N2\cdot c^2\frac{(\log N)^2}{N}\right)^{1/2}
\ge c\log N
\]
for all $x\in B_N$.
Consequently,
\[
\|\vec R f_N\|_{L^p(\Omega_N;\ell^2_N)}
\ge c\log N\,\mu_N(B_N)^{1/p}.
\]
The central binomial estimate gives
\[
\mu_N(B_N)=2^{-N}\binom{N}{N/2}\ge c N^{-1/2}.
\]
Hence, for $p\ge2$,
\begin{equation}\label{eq:sharp-Lp-lower-before-choice}
\|\vec R f_N\|_{L^p(\Omega_N;\ell^2_N)}
\ge c\log N\,N^{-1/(2p)}.
\end{equation}

Given $p\ge2$, choose $N$ even with
\[
64e^{2p}\le N\le 64e^{2p}+2.
\]
Then $N\simeq e^{2p}$,
\[
\log N\ge 2p,
\]
and
\[
N^{-1/(2p)}\ge c
\]
with a universal constant $c>0$. From \eqref{eq:sharp-Lp-lower-before-choice},
\[
\|\vec R f_N\|_{L^p(\Omega_N;\ell^2_N)}\ge c p.
\]
Since $0\le f_N\le1$, we have $\|f_N\|_{L^p(\Omega_N)}\le1$, and therefore
\[
\frac{\|\vec R f_N\|_{L^p(\Omega_N;\ell^2_N)}}{\|f_N\|_{L^p(\Omega_N)}}
\ge c p.
\]
This proves Theorem~\ref{thm:sharpness}. If one wants a mean-zero example, replace $f_N$ by $f_N-\mathbb E_{\mu_N}f_N$; the Riesz transforms annihilate constants and the $L^p$ norm remains bounded by $1$.
\end{proof}

\section*{Acknowlegments}

P.I. acknowledges partial support from the NSF CAREER grant DMS-
2152401, NSF grant DMS- 2554183, a Simons Fellowship, and a Humboldt
Research Fellowship for Experienced Researchers.
S.P. acknolewdges support from the Alexander von Humboldt Stiftung.
A.V. was supported by NSF grant DMS-2154402. We thank Ramon van Handel for helpful discussions, in particular concerning the two-point inequality used in Section~\ref{two2}.

\bibliography{Bellman_Riesz.bib}
\bibliographystyle{alpha} 

\end{document}